\documentclass[a4paper,12pt,final,leqno,notitlepage]{article}

\usepackage{amsmath}
\usepackage{amsfonts}
\usepackage{amssymb}
\usepackage{amsthm} 

\usepackage{graphicx}

\usepackage[blocks]{authblk}

\usepackage[T1]{fontenc}

\usepackage{fullpage}

\usepackage{enumerate}
\usepackage{enumitem}

\def\CC{\mathbb{C}} 
\def\DD{\mathbb{D}} 
\def\RR{\mathbb{R}} 
\def\TT{\mathbb{T}} 
\def\SS{\mathbb{S}} 
\def\HHL{\mathbb{H}_{-}} 
\def\CDD{\overline{\DD}} 
\def\OO{\mathcal{O}} 
\def\CL{\mathcal{C}} 
\def\RE{\textnormal{Re}\,} 
\def\IM{\textnormal{Im}\,} 
\def\INT{\textnormal{int}\,} 

\def\SPAN{\textnormal{span}\,}
\def\PDER#1#2{\frac{\partial #1}{\partial #2}} 
\def\DDER#1#2{\frac{d #1}{d #2}} 

\def\LEBT{\mathcal{L}^{\TT}} 
\def\DLEBT{d\mathcal{L}^{\TT}} 

\def\BLHL#1{\bar\lambda h_{#1}(\lambda)} 

\newtheoremstyle{remarkstyle}{}{}{}{}{\bf}{.}{ }{}

\newtheorem{THE}{Theorem}[section]
\newtheorem{PROP}[THE]{Proposition}
\newtheorem{LEM}[THE]{Lemma}

\newtheorem{OBS}[THE]{Observation}

\theoremstyle{remarkstyle}

\newtheorem{EX}[THE]{Example}
\newtheorem{DEF}[THE]{Definition}

\begin{document}

\renewcommand{\thepage}{\small\arabic{page}}
\renewcommand{\thefootnote}{(\arabic{footnote})}
\renewcommand{\subsection}{\arabic{section}}
\renewcommand{\thesubsection}{\arabic{subsection}}

\renewcommand\Affilfont{\small}



\author{Sylwester Zaj\k{a}c}
\affil{Institute of Mathematics, Faculty of Mathematics and Computer Science,\\ Jagiellonian University, \L ojasiewicza 6, 30-348 Krak\'ow, Poland\\ sylwester.zajac@im.uj.edu.pl}

\title{Complex geodesics in convex tube domains}

\date{}

\maketitle

\begin{abstract}
%
We describe all complex geodesics in convex tube domains.
In the case when the base of a convex tube domain does not contain any real line, the obtained description involves the notion of boundary measure of a holomorphic map and it is expressed in the language of real Borel measures on the unit circle.
Applying our result, we calculate all complex geodesics in convex tube domains with unbounded base covering special class of Reinhardt domains.
\end{abstract}




\section{Introduction}\label{sect_introduction}

A domain $D\subset\CC^n$ is called a tube if $D$ is of the form $\Omega+i\RR^n$ for some domain $\Omega\subset\RR^n$, called the base of $D$.
Tube domains play an important role in studies of Reinhardt domains, as any Reinhardt domain contained in $(\CC_*)^n$ admits a natural covering by a tube domain via the mapping $(z_1,\ldots,z_n)\mapsto (e^{z_1},\ldots,e^{z_n})$.
What is more, pseudoconvex Reinhardt domains contained in $(\CC_*)^n$ are exactly those which are covered by convex tubes.

We are interested in tube domains mostly from the point of view of holomorphically invariant distances and the Lempert theorem.
It is known that if $G\subset\CC^n$ is a pseudoconvex Reinhardt domain and $D\subset\CC^n$ is a convex tube covering $G\cap (\CC_*)^n$, then any $\ell_G$-extremal disc ($\ell_G$ is the Lempert function for $G$) which does not intersect the axes can be lifted to a complex geodesic in $D$.
As a consequence, if we know the form of all complex geodesics in $D$, we obtain a form of all $\ell_G$-extremal discs not intersecting the axes (more precisely, we get a necessary condition for a map to be a $\ell_G$-extremal, as not every geodesic in $D$ produces extremal disc in $G$).

The problem of characterising complex geodesics for convex tubes may be reduced to the case of taut convex tubes (equivalently: hyperbolic convex tubes, convex tubes without real lines contained in the base; see \cite[Theorem 1.1]{braccisarcco}).
This is a consequence of the fact that any convex tube in $\CC^n$ is linearly biholomorphic to a cartesian product of a taut convex tube and some $\CC^k$ (see Observation \ref{obs_only_bounded_from_right}).
If $D\subset\CC^n$ is a taut convex tube domain and $\varphi:\DD\to D$ is a holomorphic map, then radial limits of $\varphi$ exist almost everywhere and - what is important - $\varphi$ admits a boundary measure (see Section \ref{sect_preliminaries} for details).

In Section \ref{sect_characterisations} we give two descriptions of complex geodesics in taut convex tubes in $\CC^n$.
A starting point for us is the characterisation for bounded convex domains presented in \cite{roydenwong} and \cite[Subsection 8.2]{jarnickipflug} - it states that a holomorphic map $\varphi:\DD\to D$ is a complex geodesic for a bouded convex domain $D\subset\CC^n$ if and only if there exists a map $h:\DD\to\CC^n$ of class $H^1$ such that $\RE\left[\bar\lambda h^*(\lambda)\bullet(z-\varphi^*(\lambda))\right]<0$ for all $z\in D$ and almost every $\lambda\in\TT$.
It does not work in the case when $D$ is a convex tube with unbounded base, because the condition just mentioned is not sufficient for $\varphi$ to be a geodesic - even in the simpliest case when $D$ is a left half-plane in $\CC$.
We found it possible to add to this condition another one to obtain the equivalence (Theorem \ref{th_wkw_radial_limits}; the map $h$ takes then the form $\bar a\lambda^2+b\lambda+a$, because $D$ is a tube).
Unfortunately, the condition added is not helpful for us in calculating complex geodesics, as it is not a 'boundary condition' (it refers not only to boundary properties of $\varphi$), while the condition with radial limits seemed to be too weak, because generally radial limits of a map which is not of class $H^1$ do not give much information about it.
Our main idea was to replace these two conditions with another one, making use of the notion of boundary measures of holomorphic maps.
We obtain the following characterisation of geodesics, expressed in the language of real Borel measures on the unit circle:

\begin{THE}\label{th_wkw_miara}
Let $D\subset\CC^n$ be a taut convex tube domain and let $\varphi:\DD\to D$ be a holomorphic map with the boundary measure $\mu$.
Then $\varphi$ is a complex geodesic for $D$ iff there exists a map $h:\CC\to\CC^n$ of the form $\bar a\lambda^2+b\lambda+a$ with some $a\in\CC^n$, $b\in\RR^n$, such that $h\not\equiv 0$ and the measure $$\BLHL{}\bullet(\RE z\,\DLEBT(\lambda)-d\mu(\lambda))$$ is negative for every $z\in D$.
\end{THE}

\noindent
The above description refers only to boundary properties of the map $\varphi$ (i.e. only to its boundary measure), what makes it really helpful in calculating complex geodesics for some tube domains with unbounded base.
We present how it works in Section \ref{sect_examples}, especially in our main example - for convex tubes in $\CC^2$ covering finite intersections of Reinhardt domains of the form $$\lbrace (z_1,z_2)\in\DD^2: 0< |z_1|^p |z_2|^q < \alpha\rbrace$$ with some $p,q>0$, $\alpha\in (0,1)$ (Example \ref{ex_klis_domain}).

\section{Preliminaries}\label{sect_preliminaries}

Let us begin with some notation: $\DD$ is the unit disc in $\CC$, $\TT$ is the circle $\partial\DD$, $\HHL$ is the left half-plane $\lbrace z\in\CC:\RE z<0\rbrace$, $\SS$ is the strip $\lbrace z\in\CC:\RE z\in(0,1)\rbrace$, $\LEBT$ is the Lebesgue measure on $\TT$, and $\delta_{\lambda_0}$ is the Dirac delta at a point $\lambda_0\in\TT$.
By $T_c$, $c\in\DD$, we denote the automorphism $\lambda\mapsto\frac{\lambda-c}{1-\bar c\lambda}$ of $\DD$.
For $z,w\in\CC^n$, $\langle z,w\rangle$ is the hermitian inner product in $\CC^n$, and $z\bullet w$ is the dot product, i.e. $z\bullet w=\langle z,\bar w\rangle$.
Vectors from $\CC^n$ are identified with vertical matrices $n\times 1$, and hence $z\bullet w=z^T\cdot w$, where for a matrix $A$ the symbol $A^T$ denotes the transpose of $A$ and $\cdot$ is the standard matrix multiplication. 
For a holomorphic map $\varphi:\DD\to\CC^n$, by $\varphi^*(\lambda)$ we denote the radial limit $\lim_{r\to 1^-}\varphi(r\lambda)$ of $\varphi$ at a point $\lambda\in\TT$, whenever it exists.
Finally, $\CL(\TT)$ is the space of all complex continuous functions on $\TT$, equipped with the supremum norm, and $H^p$, $p\in[1,\infty]$, is the Hardy space on the unit disc.

Let $D\subset\CC^n$ be a domain and let $\varphi:\DD\to D$ be holomorphic.
The map $\varphi$ is called a \emph{complex geodesic} in $D$ if $\varphi$ is an isometry with respect to the Poincar\'{e} distance in $\DD$ and the Carath\'{e}odory (pseudo)distance in $D$.
A holomorphic function $f:D\to\DD$ such that $f\circ\varphi=id_{\DD}$ is called a \emph{left inverse} of $\varphi$.
The map $\varphi$ is a complex geodesic in $D$ iff it admits a left inverse on $D$.
By the Lempert theorem, if $D$ is a taut convex domain, then for any pair of points in $D$ there exists a complex geodesic passing through them (see \cite{lempert} or \cite[Chapter 8]{jarnickipflug} for details).
%

\begin{DEF}\label{def_bounded_from_the_right}
We say that a domain $D\subset\CC^n$ is a convex tube if $D=\Omega+i\RR^n$ for some convex domain $\Omega\subset\RR^n$.
We call $\Omega$ the \emph{base} of $D$ and we denote it by $\RE D$.
\end{DEF}

As a quite easy consequence of \cite[Theorem 1.1, Propositions 1.2 and 3.5]{braccisarcco} we obtain the following decomposition:

\begin{OBS}\label{obs_only_bounded_from_right}
Let $D\subset\CC^n$ be a convex tube domain.
Then there exist a number $k\in\lbrace 0,\ldots n\rbrace$, a convex tube $G\subset\HHL^k$ and a complex affine isomorphism $\Phi$ of $\CC^n$ such that $\Phi(\RR^n)=\RR^n$ and $\Phi(D)=G\times\CC^{n-k}$.
Moreover, a holomorphic mapping $\varphi:\DD\to D$ is a complex geodesic for $D$ iff $(\Phi_1,\ldots,\Phi_k)\circ\varphi$ is a complex geodesic for $G$.
\end{OBS}

The number $n-k$ is equal to the maximal dimension of a real affine subspace contained in $\RE D$.
In view of the above observation, it is enough to restrict our considerations to taut convex tubes.
If $D$ is a taut convex tube, then $k=n$ and $\Phi(D)\subset\HHL^n$.
Moreover, if $\varphi:\DD\to D$ is a holomorphic map, then the non-tangential limit $\varphi^*(\lambda)$ exists end belongs to $\overline{D}$ for almost every $\lambda\in\TT$.
As we shall see later, $\varphi$ admits also a boundary measure.

Let us recall some facts connected with complex measures on the unit circle.
Below, we consider only Borel measures on $\TT$ so we shall usually omit the word 'Borel'.
It is known that any finite positive measure on $\TT$ is regular in the sense that the measure of any Borel subset $A\subset\TT$ may be approximated by the measures of both compact subsets and open supersets of $A$.
Hence, any complex measure on $\TT$ is regular, i.e. its variation is a regular measure.
In view of the Riesz representation theorem, complex measures on $\TT$ may be identified with continuous linear functionals on $\CL(\TT)$.

We shall use the symbols $\langle\cdot,\cdot\cdot\rangle$ and $\bullet$ also for measures and functions, e.g. if $\mu$ is a tuple $(\mu_1,\ldots,\mu_n)$ of complex measures and $v=(v_1,\ldots,v_n)$ is a vector or a bounded Borel-measurable mapping on $\TT$, then $\langle d\mu,v\rangle$ is the measure $\sum_{j=1}^n \bar{v_j} d\mu_j$, and $v\bullet d\mu$ is the measure $\sum_{j=1}^n  v_j d\mu_j$, etc.
The fact that a real measure $\nu$ is positive (resp. negative, null) is shortly denoted by $\nu\geq 0$ (resp. $\nu\leq 0$, $\nu=0$).

Introduce the family $$\mathcal{M}:=\lbrace f_\mu+i\alpha:\mu\text{ is a real measure on }\TT,\alpha\in\RR\rbrace,$$ where $f_\mu:\DD\to\CC$ is the holomorphic function given by
$$f_\mu(\lambda)=\frac{1}{2\pi}\int_{\TT}\frac{\zeta+\lambda}{\zeta-\lambda}d\mu(\zeta),\;\lambda\in\DD$$
(by a real measure we mean a complex measure with values in $\RR$).
It is known (see e.g. \cite[p. 10]{koosis}) that the measures $\RE f_\mu(r\cdot)\DLEBT$ tend weakly-* to $\mu$ when $r\to 1^-$ (as continuous linear functionals on $\CL(\TT)$, i.e. $\int_{\TT}u(\lambda)\RE f_\mu(r\lambda)\DLEBT(\lambda)\to\int_{\TT}u(\lambda)d\mu(\lambda)$ for any $u\in\CL(\TT)$) and $\mu$ is uniquely determined by $f_\mu$.
Thus, any $f\in\mathcal{M}$ has a unique decomposition $f=f_\mu+i\alpha$; then we call $\mu$ the \emph{boundary measure} of $f$.
One can check, that $\RE f\geq 0$ on $\DD$ iff its boundary measure is positive.

If a holomorphic function $f:\DD\to\CC$ is of class $H^1$, then it belongs to $\mathcal{M}$, its boundary measure is just $\RE f^*\DLEBT$ and the functions $f(r\cdot)$ tend to $f^*$ in the $L^1$ norm with respect to the measure $\LEBT$ (see e.g. \cite[p. 35]{koosis}).

By $\mathcal{M}^n$ we denote the set of all $\varphi=(\varphi_1,\ldots,\varphi_n)\in\OO(\DD,\CC^n)$ with $\varphi_1,\ldots,\varphi_n\in\mathcal{M}$.
In this situation, by the boundary measure of $\varphi$ we mean the $n$-tuple $\mu=(\mu_1,\ldots,\mu_n)$ of measures with each $\mu_j$ being the boundary measure of $\varphi_j$.
If $V$ is a real $m\times n$ matrix and $b\in\RR^m$, then the map $\lambda\mapsto V\cdot\varphi(\lambda)+b$ belongs to $\mathcal{M}^m$ and its boundary measure is just $V\cdot\mu+b\,d\LEBT$.

The Herglotz Representation Theorem (see e.g. \cite[p. 5]{koosis}) states that any $f\in\OO(\DD,\CC)$ with non-negative real part belongs to $\mathcal{M}$, and hence $\OO(\DD,\HHL^n)\subset\mathcal{M}^n$.
As a consequence of this fact and Observation \ref{obs_only_bounded_from_right}, we deduce that for any taut convex tube domain $D\subset\CC^n$ the family $\OO(\DD,D)$ is contained in $\mathcal{M}^n$, so any holomorphic map $\varphi:\DD\to D$ admits a boundary measure.

Let us emphasize that the Poisson formula
\begin{equation}\label{eq_poisson_formula}
\varphi(\lambda)=\frac{1}{2\pi}\int_{\TT}\frac{\zeta+\lambda}{\zeta-\lambda}d\mu(\zeta)+i\IM\varphi(0)\;\lambda\in\DD
\end{equation}
(the integral with respect to the tuple $\mu=(\mu_1,\ldots,\mu_n)$ is just the tuple of integrals with respect to $\mu_1,\ldots,\mu_n$) may be stated in terms of the boundary measure $\mu$ of $\varphi$, while it cannot be stated using only its radial limits (e.g. for $\varphi(\lambda)=\frac{1+\lambda}{1-\lambda}$ we have $\RE \varphi^*(\lambda)=0$ for a.e. $\lambda\in\TT$, while $\mu=2\pi\delta_1$).
This is an advantage of boundary measures over radial limits and the reason for which Theorem \ref{th_wkw_miara} is formulated in the language of the measure theory.

\section{Characterisations of complex geodesics}\label{sect_characterisations}

In this section we show two characterisations of complex geodesics in taut convex tubes.
The first one obtained in Theorem \ref{th_wkw_radial_limits} is expressed in terms of radial limits of a map and it is similar to the characterisation in the case of bounded convex domains presented in \cite{roydenwong} and \cite[Subsection 8.2]{jarnickipflug}, while the second one, expressed in the language of measures, is formulated in Theorem \ref{th_wkw_miara}.
Although calculating geodesics we mainly use Theorem \ref{th_wkw_miara}, the conditions appearing in Theorem \ref{th_wkw_radial_limits} are useful in deriving some additional information, as they are 'connected' with those in Theorem \ref{th_wkw_miara} via Lemma \ref{lem_wkw_h_fi_miara}.

Let us begin with a sufficient condition:

\begin{PROP}\label{prop_warunki_wystarczajace}
Let $D\subset\CC^n$ be a taut convex tube domain, and let $\varphi:\DD\to D$ be a holomorphic map.
Suppose that there exists a mapping $h:\CC\to\CC^n$ of the form $h(\lambda)=\bar a \lambda^2+b\lambda+a$ with some $a\in\CC^n,b\in\RR^n$, such that:
\begin{enumerate}
\renewcommand{\theenumi}{(\roman{enumi})}
\renewcommand{\labelenumi}{\theenumi}
\item\label{prop_warunki_wystarczajace_1} $\RE\left[\bar\lambda h(\lambda)\bullet(z-\varphi^*(\lambda))\right]<0$ for all $z\in D$ and a.e. $\lambda\in\TT$\footnote{That is, the statement holds for $(z,\lambda)\in D\times A$, where $A\subset\TT$ is a Borel subset of the full $\LEBT$ measure.},
\item\label{prop_warunki_wystarczajace_2} $\RE\left[h(\lambda)\bullet\frac{\varphi(0)-\varphi(\lambda)}{\lambda}\right]<0$ for every $\lambda\in\DD_*$.
\end{enumerate}
Then $\varphi$ is a complex geodesic for $D$.
\end{PROP}

By \cite[Lemma 8.2.2]{jarnickipflug}, in the case of a bounded domain $D$, if there exists some $h:\DD\to\CC^n$ of class $H^1$ satisfying \ref{prop_warunki_wystarczajace_1} (with $h^*(\lambda)$ instead of $h(\lambda)$), then $\varphi$ admits a left inverse on $D$.
In our situation, i.e. when $D$ is a taut convex tube domain, one can show (putting $z+isx$ instead of $z$ in \ref{prop_warunki_wystarczajace_1}, with $s\in\RR$ and fixed $x\in\RR^n$, $z\in D$) that such a function $h$ satisfies $\BLHL{}\in\RR^n$ for a.e. $\lambda\in\TT$, and hence it must be of the form $h(\lambda)=\bar a\lambda^2+b\lambda+a$ (see e.g. \cite[Lemma 2]{gentili}).
However, this assumption is no longer sufficient for $\varphi$ to admit a left inverse (and hence to be a complex geodesic for $D$).
For example, take $D=\HHL$ and $\varphi(\lambda)=\frac{\lambda^2+1}{\lambda^2-1}$.
One can easily check that $\varphi$ satisfies \ref{prop_warunki_wystarczajace_1} with $h(\lambda)=\lambda$, but clearly it is not a complex geodesic for $D$ and it does not fulfill \ref{prop_warunki_wystarczajace_2}.

Actually, the proof of Proposition \ref{prop_warunki_wystarczajace} is very similar to the proof of \cite[Lemma 8.2.2]{jarnickipflug}.
The only trouble appears when we need to use a version of maximum principle for harmonic functions: knowing that $u$ is harmonic on $\DD$ and $u^*<0$ a.e. on $\TT$, generally we cannot conclude that $u<0$ in $\DD$ (in particular, \ref{prop_warunki_wystarczajace_2} does not follow from \ref{prop_warunki_wystarczajace_1} applied for $z=\varphi(0)$); something more of $u$ need to be assumed, e.g. that $u$ is bounded from above.
This is the reason for which the condition \ref{prop_warunki_wystarczajace_2} appears.
For the reader's convenience, we show the proof of Proposition \ref{prop_warunki_wystarczajace} in section \ref{sect_appendix}.
Actually, based on the proof of \cite[Lemma 8.2.2]{jarnickipflug} we state a lemma which is somewhat more general than Proposition \ref{prop_warunki_wystarczajace} (it works for every domain in $\CC^n$) and which gives that proposition as an immediate corollary.
The lemma is formulated in a bit different form than the proposition, for reasons explained in Section \ref{sect_appendix}.

Let us remark that if $\RE D$ is bounded, then the condition \ref{prop_warunki_wystarczajace_2} in Proposition \ref{prop_warunki_wystarczajace} can be omitted.
This follows from the fact that $\RE \varphi$ is bounded and hence the maps $\varphi$ and $\lambda\mapsto h(\lambda)\bullet\frac{\varphi(0)-\varphi(\lambda)}{\lambda}$ are of class $H^1$ (with $h$ as in Proposition \ref{prop_warunki_wystarczajace}), so the maximum principle can be applied to deduce \ref{prop_warunki_wystarczajace_2} from \ref{prop_warunki_wystarczajace_1}.

Let us now state necessary conditions for a map $\varphi$ is a complex geodesic:

\begin{PROP}\label{prop_warunki_konieczne}
Let $D\subset\CC^n$ be a taut convex tube domain, let $\varphi:\DD\to D$ be a complex geodesic and let $f:D\to\DD$ be a left inverse for $\varphi$.
Define $$h(\lambda):=\left(\PDER{f}{z_1}(\varphi(\lambda)),\ldots,\PDER{f}{z_n}(\varphi(\lambda))\right),\;\lambda\in\DD.$$
Then:
\begin{enumerate}
\renewcommand{\theenumi}{(\roman{enumi})}
\renewcommand{\labelenumi}{\theenumi}
\item\label{prop_warunki_konieczne_1} $h(\lambda)=\bar a \lambda^2+b\lambda+a$ for some $a\in\CC^n,b\in\RR^n$, and $h\not\equiv 0$,
\item\label{prop_warunki_konieczne_2} $\RE\left[\bar\lambda h(\lambda)\bullet(z-\varphi^*(\lambda))\right]<0$ for all $z\in D$ and a.e. $\lambda\in\TT$,
\item\label{prop_warunki_konieczne_3} $\RE\left[h(\lambda)\bullet\frac{\varphi(0)-\varphi(\lambda)}{\lambda}\right]<0$ for every $\lambda\in\DD_*$.
\end{enumerate}
\end{PROP}

\noindent
Propositions \ref{prop_warunki_wystarczajace} and \ref{prop_warunki_konieczne} let us state the following characterisation of geodesics:

\begin{THE}\label{th_wkw_radial_limits}
Let $D\subset\CC^n$ be a taut convex tube domain and let $\varphi:\DD\to D$ be a holomorphic map.
Then $\varphi$ is a complex geodesic for $D$ iff there exists a mapping $h:\CC\to\CC^n$ of the form $\bar a\lambda^2+b\lambda+a$ with some $a\in\CC^n, b\in\RR^n$, such that:
\begin{enumerate}
\renewcommand{\theenumi}{(\roman{enumi})}
\renewcommand{\labelenumi}{\theenumi}
\item\label{th_wkw_radial_limits_1} $\RE\left[\bar\lambda h(\lambda)\bullet(z-\varphi^*(\lambda))\right]<0$ for all $z\in D$ and a.e. $\lambda\in\TT$,
\item\label{th_wkw_radial_limits_2} $\RE\left[h(\lambda)\bullet\frac{\varphi(0)-\varphi(\lambda)}{\lambda}\right]<0$ for every $\lambda\in\DD_*$.
\end{enumerate}
\end{THE}

\begin{proof}[Proof of Proposition \ref{prop_warunki_konieczne}]
Since $f(\varphi(\lambda))=\lambda$, we have
\begin{equation}\label{eq_pwk_1145}
h(\lambda)\bullet\varphi'(\lambda)=1,\;\lambda\in\DD.
\end{equation}
In particular, $h\not\equiv 0$.

Denote $$f_{z,t}(\lambda):=f((1-t)\varphi(\lambda)+tz),\;\lambda\in\DD,t\in[0,1],z\in D.$$
There is $f_{z,0}(\lambda)=\lambda$.
We have
\begin{equation*}
\left.\DDER{|f_{z,t}(\lambda)|^2}{t}\right|_{t=0}=2\,\RE\left[\BLHL{}\bullet(z-\varphi(\lambda))\right].
\end{equation*}
On the other hand, as $f_{z,t}\in\OO(\DD,\DD)$, by \cite[Lemma 1.2.4]{abate} we have the inequality $$|f_{z,t}(\lambda)|-|\lambda|\leq\frac{2|f_{z,t}(0)|}{1+|f_{z,t}(0)|}(1-|\lambda|).$$
Therefore, $$\frac{|f_{z,t}(\lambda)|^2-|\lambda|^2}{t} \leq 2\,\frac{|f_{z,t}(\lambda)|-|\lambda|}{t} \leq \frac{4|\frac1t f_{z,t}(0)|}{1+|f_{z,t}(0)|}(1-|\lambda|),$$
so $$\left.\DDER{|f_{z,t}(\lambda)|^2}{t}\right|_{t=0} \leq 4(1-|\lambda|)\left|\left.\DDER{f_{z,t}(0)}{t}\right|_{t=0}\right| \leq 4(1-|\lambda|)\,|h(0)\bullet(z-\varphi(0))|.$$
In summary, we obtain
\begin{equation}\label{eq_pwk_1234}
\RE\left[\BLHL{}\bullet(z-\varphi(\lambda))\right] \leq 2(1-|\lambda|)\,|h(0)\bullet(z-\varphi(0))|,\;\lambda\in\DD,z\in D.
\end{equation}

Putting $z=\varphi(0)$ we obtain the weak inequality in \ref{prop_warunki_konieczne_3}.
The strong one follows from the maximum principle for the harmonic function $$\DD\ni\lambda\mapsto\RE\left[h(\lambda)\bullet\frac{\varphi(0)-\varphi(\lambda)}{\lambda}\right]$$
- it is non-constant, because its value at $\lambda=0$ equals to $-\RE\left[h(0)\bullet\varphi'(0)\right]=-1$, by (\ref{eq_pwk_1145}).

Putting $z=\varphi(0)+ise_j$ in (\ref{eq_pwk_1234}), where $j\in\lbrace 1,\ldots,n\rbrace$ and $s\in\RR$, we get $$\RE\left[\BLHL{}\bullet(\varphi(0)-\varphi(\lambda))\right] \leq \IM(\bar\lambda h_j(\lambda))s + 2(1-|\lambda|)\,|h_{j}(0)||s|.$$
Hence, for fixed $\lambda\in\DD$ the function of variable $s$ on the right side is bounded from below.
This implies
\begin{equation}\label{eq_pwk_1133}
|\IM(\bar\lambda h_j(\lambda))| \leq 2(1-|\lambda|)\,|h_j(0)|,\;\lambda\in\DD.
\end{equation}
Writing $h_j(\lambda)= h_j(0)+\lambda g_j(\lambda)$ we obtain $$|\lambda|^2|\IM g_j(\lambda)|-|\IM(\bar\lambda h_j(0))|\leq 2\,(1-|\lambda|)|h_j(0)|.$$
By the maximum principle, $\IM g_j$ is bounded, so $g_j$ and $h_j$ are of class $H^1$.
Tending with $\lambda$ non-tangentially to $\TT$ in (\ref{eq_pwk_1133}) we obtain $\IM (\bar\lambda h_j^*(\lambda))=0$ a.e. on $\TT$.
This implies $$\IM(g_j^*(\lambda)-\overline{h_j(0)}\lambda) = \IM(\bar\lambda h_j^*(\lambda)) = 0\;\text{ for a.e. }\lambda\in\TT,$$
so $g_j(\lambda)-\overline{h_j(0)}\lambda$ is equal to some real constant $b_j$, as $g_j$ is of class $H^1$.
We get \ref{prop_warunki_konieczne_1} (and we can extend $h$ to the whole $\CC$).

Tending with $\lambda$ non-tangentially to $\TT$ in (\ref{eq_pwk_1234}) we obtain the weak inequality in \ref{prop_warunki_konieczne_2}.
The strong inequality follows from the fact that for a.e. $\lambda\in\TT$ the mapping $$D\ni z\mapsto \RE\left[\bar\lambda h(\lambda)\bullet(z-\varphi^*(\lambda))\right]\in\RR$$ is affine (over $\RR$), non-constant, and hence open.
\end{proof}

Note that we can obtain the statement \ref{prop_warunki_konieczne_1} in Proposition \ref{prop_warunki_konieczne} immediately, using more general fact - see \cite[Theorem 3]{edigarianzwonek}.

The fact that the condition \ref{th_wkw_radial_limits_2} from Theorem \ref{th_wkw_radial_limits} is not a 'boundary' condition makes it not very useful when we want to compute complex geodesics for a given tube domain.
On the other hand, the condition with radial limits seems to be too weak in the case of unbounded base of $D$, because radial limits of a map generally do not give us much information about it.
Fortunately, making use of bounadry measures we can replace these two conditions by another one, expressed in the language of the measure theory and referring only to boundary properties of a mapping (Theorem \ref{th_wkw_miara}).
These new condition turns out to be useful for calculating complex geodesics for some convex tubes with unbounded base.

Let us note that in the case when $\RE D$ is bounded, the boundary measure of a map $\varphi$ is just $\RE\varphi^*\DLEBT$ and we have the Poisson formula for $\varphi$ with its radial limits.
Therefore, for such $D$ the conditions with radial limits seems to be sufficient for our purposes (see e.g. Example \ref{ex_kolo}).

Theorem \ref{th_wkw_miara} is a consequence of Theorem \ref{th_wkw_radial_limits} and the following lemma:

\begin{LEM}\label{lem_wkw_h_fi_miara}
Let $D\subset\CC^n$ be a taut convex tube, let $\varphi:\DD\to D$ be a holomorphic map with the boundary measure $\mu$, and let $h(\lambda)=\bar a\lambda^2+b\lambda+a$, $\lambda\in\DD$, for some $a\in\CC^n$, $b\in\RR^n$, with $h\not\equiv 0$.
Then
\begin{align*}\label{eq_lem_measures}\tag{m}
\text{the measure }\BLHL{}\bullet(\RE z\,\DLEBT(\lambda)-d\mu(\lambda))\text{ is negative for every }z\in D
\end{align*}
iff the following two conditions holds:
\begin{enumerate}
\renewcommand{\theenumi}{(\roman{enumi})}
\renewcommand{\labelenumi}{\theenumi}
\item\label{th_wkw_miara_1} $\RE\left[\bar\lambda h(\lambda)\bullet(z-\varphi^*(\lambda))\right]<0$ for all $z\in D$ and a.e. $\lambda\in\TT$,
\item\label{th_wkw_miara_2} $\RE\left[h(\lambda)\bullet\frac{\varphi(0)-\varphi(\lambda)}{\lambda}\right]<0$ for every $\lambda\in\DD_*$.
\end{enumerate}
\end{LEM}

All measures in \eqref{eq_lem_measures} are regular and real.
Let us also note that in view of this lemma, the function $h$ in Theorem \ref{th_wkw_miara} is the same $h$ as in Theorem \ref{th_wkw_radial_limits} - we shall use this fact in Section \ref{sect_examples}.

\begin{proof}[Proof of Lemma \ref{lem_wkw_h_fi_miara}]
We start with showing that the condition \ref{th_wkw_miara_1} $\wedge$ \ref{th_wkw_miara_2} is equivalent to the following:
\begin{equation}\label{eq_lem_wkw_psi_1111}
\RE\psi_z(\lambda)\leq 0\;\text{for all }\lambda\in\DD,z\in D,
\end{equation}
where $\psi_z:\DD\to\CC$ is the holomorphic function defined as
\begin{equation}\label{eq_lem_wkw_psi_2277}
\psi_z(\lambda) = \frac{\varphi(0)-\varphi(\lambda)}{\lambda}\bullet h(\lambda)+\frac{h(\lambda)-h(0)}{\lambda}\bullet(z-\varphi(0))+\lambda\,\overline{h(0)\bullet(z-\varphi(0))}
\end{equation}
for $z\in D$, $\lambda\in\DD_*$ (and extended holomorphically through $0$).
Indeed, one can check that
\begin{equation}\label{eq_lem_wkw_880123}
\RE\psi^*_z(\lambda)=\RE\left[\bar\lambda h(\lambda)\bullet(z-\varphi^*(\lambda))\right]\text{ for all }z\in D\text{ and a.e. }\lambda\in\TT.
\end{equation}
Now, \ref{th_wkw_miara_1} $\wedge$ \ref{th_wkw_miara_2} implies that $\RE\psi_z$ is bounded from above and $\RE\psi_z^*< 0$ a.e. on $\TT$, so the maximum principle easily gives \eqref{eq_lem_wkw_psi_1111}.
On the other hand, \eqref{eq_lem_wkw_psi_1111} let us derive the weak inequalities in \ref{th_wkw_miara_1} and \ref{th_wkw_miara_2}.
The strong inequality in \ref{th_wkw_miara_1} follows from the fact that the map $z\mapsto\RE\left[\bar\lambda h(\lambda)\bullet(z-\varphi^*(\lambda))\right]$ is open for a.e. $\lambda\in\TT$ (as $h\not\equiv 0$), and the strong inequality in \ref{th_wkw_miara_2} is a consequence of the maximum principle, because by \ref{th_wkw_miara_1} with $z=\varphi(0)$ the function $\lambda\mapsto\RE\left[h(\lambda)\bullet\frac{\varphi(0)-\varphi(\lambda)}{\lambda}\right]$ is not identically equal to $0$.

Let $\nu_z$ denote the measure in the condition \eqref{eq_lem_measures}.
To finish the proof, it suffices to show that the conditions \eqref{eq_lem_wkw_psi_1111} and \eqref{eq_lem_measures} are equivalent, and for this it is enough to prove that $\psi_z\in\mathcal{M}$ and $\nu_z$ is the boundary measure of $\psi_z$.

We claim that the Poisson formula holds for $\RE\psi_z$ and $\nu_z$, i.e.
\begin{equation}\label{eq_lem_wkw_2157}
\RE\psi_z(\lambda) = \frac{1}{2\pi}\int_{\TT}\frac{1-|\lambda|^2}{|\zeta-\lambda|^2}d\nu_z(\zeta),\;\lambda\in\DD,z\in D
\end{equation}
(this shall finish the proof, in view of the definition of $\mathcal{M}$).
Fix $z\in D$.
Write
\begin{eqnarray*}
\nu_z = \BLHL{} \bullet (\RE z\,\DLEBT(\lambda)-d\mu(\lambda)) &=& \BLHL{}\bullet\left(\RE\varphi(0)\DLEBT(\lambda)-d\mu(\lambda)\right)\\
&+& \RE\left[\bar\lambda\left(h(\lambda)-h(0)\right)\bullet\left(z-\varphi(0)\right)\right]\DLEBT(\lambda)\\
&+& \RE\left[\bar\lambda h(0)\bullet\left(z-\varphi(0)\right)\right]\DLEBT(\lambda)
\end{eqnarray*}
(remember, that $\bar\lambda h(\lambda)\in\RR$ for $\lambda\in\TT$).
The function $\RE\psi_z$ is clearly equal to the sum of the following three terms: $$\RE\left[\frac{\varphi(0)-\varphi(\lambda)}{\lambda}\bullet h(\lambda)\right],\; \RE\left[\frac{h(\lambda)-h(0)}{\lambda}\bullet(z-\varphi(0))\right],\; \RE\left[\lambda\,\overline{h(0)\bullet(z-\varphi(0))}\right].$$
We have
$$\BLHL{}\bullet\left(\RE\varphi(0)\DLEBT(\lambda)-d\mu(\lambda)\right) = \lim_{r\to 1^-} \RE\left[\frac{\varphi(0)-\varphi(r\lambda)}{\lambda}\bullet h(\lambda)\right]\DLEBT(\lambda)$$
(weak-* limit), so the Poisson formula with the measure on the left hand side gives the first term.
Next,
$$\RE\left[\bar\lambda\left(h(\lambda)-h(0)\right)\bullet\left(z-\varphi(0)\right)\right]\DLEBT(\lambda) = \RE\left[\frac{h(\lambda)-h(0)}{\lambda}\bullet\left(z-\varphi(0)\right)\right]\DLEBT(\lambda),$$ so here the Poisson formula gives the second term.
Finally, the formula for the measure $\RE\left[\bar\lambda h(0)\bullet\left(z-\varphi(0)\right)\right]\DLEBT(\lambda)$ clearly gives the third term.
In summary, we get \eqref{eq_lem_wkw_2157}, what finishes the proof of the Lemma.
\end{proof}

\section{Calculating complex geodesics}\label{sect_examples}

In this section we focus on calculating complex geodesics in convex tubes in $\CC^2$ covering finite intersections of Reinhardt domains of the form $$\lbrace (z_1,z_2)\in\DD^2: 0< |z_1|^p |z_2|^q < \alpha\rbrace$$ with some $p,q>0$, $\alpha\in (0,1)$ (Example \ref{ex_klis_domain}).
For this, we state two lemmas which partially describe boundary measures of geodesics in some special situations (Lemmas \ref{lemat_odcinki} and \ref{lemat_kanty}) and which are afterwards applied to calculate all complex geodesics in Example \ref{ex_klis_domain}.
 
Before we start analysing the examples, let us make a few useful remarks; below we assume that $D\subset\CC^n$ is a taut convex tube domain.

If $\varphi:\DD\to D$ is a complex geodesic and $\lambda\in\TT$ is such that $\BLHL{}\neq 0$ and the inequality $\RE\left[\BLHL{}\bullet(z-\varphi^*(\lambda))\right]<0$ holds for all $z\in D$, then $\varphi^*(\lambda)\in\partial D$ and the vector $\BLHL{}$ is outward from $D$ at $\varphi^*(\lambda)$ (and hence it is outward from $\RE D$ at $\RE \varphi^*(\lambda)$, as $\BLHL{}\in\RR^n$).
This observation is helpful in deriving some information about $h$ and $\varphi$, or even in deriving a formula for $\varphi$ in the case of bounded base of $D$ - similarly as in Example \ref{ex_kolo}; however, it is not sufficient if the base of $D$ is unbounded.

If $\varphi$ is a complex geodesic for $D$, then - by Lemma \ref{lem_wkw_h_fi_miara} - the function $h$ from Theorem \ref{th_wkw_miara} satisfies the conclusion of Theorem \ref{th_wkw_radial_limits}, and vice versa
(what is more, $h$ may be chosen as in Proposition \ref{prop_warunki_konieczne}).
In particular, given a map $h$ as in Theorem \ref{th_wkw_miara} we can apply for it the conclusions made in previous paragraph.

Let us make the following simple observation: given a finite positive measure $\nu$ on $\TT$, a non-negative continuous function $u$ on $\TT$ with $u^{-1}(\lbrace 0\rbrace)=\lbrace\lambda_1,\ldots,\lambda_m\rbrace$, if $ud\nu$ is a null measure, then $\nu=\sum_{j=1}^m\alpha_j\delta_{\lambda_j}$ for some constants $\alpha_1,\ldots,\alpha_m\geq 0$.

Recall that by \cite[Lemma 8.4.6]{jarnickipflug}, if $h\in\OO(\DD,\CC)$ is of class $H^1$ and such that $\bar\lambda h^*(\lambda)>0$ for a.e. $\lambda\in\TT$, then $h$ is of the form $c(\lambda-d)(1-\bar d\lambda)$ with some $d\in\CDD$, $c>0$.
In particular, such a function $h$ has at most one zero on $\TT$ (counting without multiplicities).
By the observation above, if $\nu$ is a finite negative measure on $\TT$ such that the measure $\BLHL{}d\nu(\lambda)$ is null, then $\nu=\alpha\delta_{\lambda_0}$ for some $\alpha\leq 0$, $\lambda_0\in\TT$, with $\alpha h(\lambda_0)=0$ (we take $\lambda_0=d$ if $d\in\TT$, otherwise $\nu$ is null and we put $\alpha=0$ with an arbitrary $\lambda_0$).
We shall quite often use this fact.

Let as also note that if for some $p,v\in\RR^n$ the inequality $\langle \RE z-p,v\rangle<0$ holds for all $z\in D$ and $\varphi:\DD\to D$ is a holomorphic map with the boundary measure $\mu$, then a similar inequality holds for measures: $\langle d\mu-p\,d\LEBT,v\rangle\leq 0$.
This is an immediate consequence of the fact that this measure is equal to the weak-* limit of the negative measures $\langle\RE\varphi(r\cdot)-p,v\rangle\,d\LEBT$, when $r\to 1^-$.
In particular, if $\RE D\subset (-\infty,0)^n$, then $\mu_1,\ldots,\mu_n\leq 0$.

We start with two simple examples: $D=\HHL^n$ and $D=\SS$ (the second one is further needed).
Afterwards, we show two lemmas and we move to Example \ref{ex_klis_domain}.

\begin{EX}\label{ex_iloczyn_lewych_polplaszczyzn}
A map $\varphi\in\mathcal{M}^n$ with the boundary measure $\mu=(\mu_1,\ldots,\mu_n)$ is a complex geodesic for the domain $\HHL^n$ iff $$\mu_{j_0}=\alpha\delta_{\lambda_0}$$ for some $j_0\in\lbrace 1,\ldots,n\rbrace$, $\alpha<0$, $\lambda_0\in\TT$.

Indeed, assume that $\varphi$ is a geodesic for $\HHL^n$ and let $h=(h_1,\ldots,h_n)$ be as in Theorem \ref{th_wkw_miara}, i.e. $h\not\equiv 0$ and $$\BLHL{}\bullet(\RE z\,\DLEBT(\lambda)-d\mu(\lambda))\leq 0,\;z\in\HHL^n.$$
Tending with $z$ to $0$ we obtain $\BLHL{}\bullet d\mu(\lambda)\geq 0$.
On the other hand, $\BLHL{j}\geq 0$ on $\TT$, because $h$ in continuous on $\TT$ and $\BLHL{}$ is outward from $\HHL^n$ for a.e. $\lambda\in\TT$, and $\mu_j\leq 0$, as $\RE\varphi_j<0$ on $\DD$.
This implies $\BLHL{}\bullet d\mu(\lambda)\leq 0$, and finally: $$\BLHL{1}d\mu_1(\lambda)+\ldots+\BLHL{n}d\mu_n(\lambda)=0.$$
Since all terms of the above sum are negative measures, we have $\BLHL{j}d\mu_j(\lambda)=0$ for every $j=1,\ldots,n$.
There exists $j_0$ such that $h_{j_0}\not\equiv 0$, and as $\mu_j$ is non-null for every $j$ (because $\RE\varphi_j\not\equiv 0$), the function $h_{j_0}$ must admit a root $\lambda_0$ on $\TT$.
Hence, we have $$\mu_{j_0}=\alpha\delta_{\lambda_0}$$ for some $\alpha<0$.
In view of \eqref{eq_poisson_formula}, the map $\varphi_{j_0}$ is given by the formula $$\varphi_{j_0}(\lambda)=\frac{\alpha}{2\pi}\,\frac{\lambda_0+\lambda}{\lambda_0-\lambda} + i\beta,\;\lambda\in\DD,$$ for some real constant $\beta$, what is a well-known form of geodesics in $\HHL^n$.
\end{EX}

\begin{EX}\label{ex_pas_w_c}
A map $\varphi\in\mathcal{M}$ with the boundary measure $\mu$ is a complex geodesic for the strip $\SS$ iff
\begin{equation}\label{eq_ex_s_geod_miara_iff}
\mu=\chi_{\lbrace\lambda\in\TT:\, \BLHL{}>0\rbrace}\,d\LEBT 
\end{equation}
for some function $h:\lambda\mapsto \bar a\lambda^2+b\lambda+a$ with $a\in\CC$, $b\in\RR$, $|b|<2|a|$  (the last condition is equivalent to $\LEBT({\lbrace\lambda\in\TT: \BLHL{}>0\rbrace})\in (0,2\pi)$).

Indeed, assume that $\varphi$ is a complex geodesic for $\SS$ and let $h$ be as in Theorem \ref{th_wkw_miara}.
The vector $\BLHL{}$ is outward from $\RE\SS=(0,1)$ at $\RE\varphi^*(\lambda)\in\partial\RE\SS=\lbrace 0,1\rbrace$ for a.e. $\lambda\in\TT$, so $\RE\varphi^*(\lambda)=1$ when $\BLHL{}>0$, and $\RE\varphi^*(\lambda)=0$ when $\BLHL{}<0$ (for a.e. $\lambda$).
Thus $\mu=\RE\varphi^*d\LEBT=\chi_{\lbrace\lambda\in\TT:\,\BLHL{}>0\rbrace}\,d\LEBT$.
As $\varphi$ is non-constant, there is $0\neq\mu\neq\LEBT$, and hence $\LEBT(\lbrace\lambda\in\TT:\,\BLHL{}>0\rbrace)\in (0,2\pi)$.

It is easy to check (using Theorem \ref{th_wkw_miara} with the same $h$ as in \eqref{eq_ex_s_geod_miara_iff}) that any map $\varphi\in\mathcal{M}$ with the boundary measure of the form \eqref{eq_ex_s_geod_miara_iff} is a geodesic for $\SS$.

Of course, formulas for geodesics in $\SS$ are well-known, and it is good to write explicitely the formula for a map $\varphi$ with the boundary measure of the form \eqref{eq_ex_s_geod_miara_iff} (we shall need it in Example \ref{ex_klis_domain}).
By the Poisson formula we have
\begin{equation*}
\varphi(\lambda) = \frac{1}{2\pi}\int_{\lbrace\zeta\in\TT:\,\bar\zeta h(\zeta)>0\rbrace}\frac{\zeta+\lambda}{\zeta-\lambda}d\LEBT(\zeta) + i\IM\varphi(0), \;\lambda\in\DD.
\end{equation*}
The mapping
\begin{equation}\label{eq_ex_s_tau}
\tau(\lambda):=-\frac i\pi \log\left(i\,\frac{1+\lambda}{1-\lambda}\right),
\end{equation}
where $\log$ denotes the branch of the logarithm with the argument in $[0,2\pi)$, is a biholomorphism from $\DD$ to $\SS$.
It extends continuously to $\CDD\setminus\lbrace -1,1\rbrace$, and it sends the arc $\lbrace\lambda\in\TT:\IM\lambda>0\rbrace$ to the line $1+i\,\RR$ and the arc $\lbrace\lambda\in\TT:\IM\lambda<0\rbrace$ to the line $i\,\RR$.
The map $$\widetilde{\varphi}(\lambda) := \tau\left(i T_c(\tfrac{\bar a}{|a|}\lambda)\right)+i\IM\varphi(0),$$
where $$c=\frac{-b}{2|a|+\sqrt{4|a|^2-b^2}},$$ is a complex geodesic for $\SS$.
There is $\IM\widetilde{\varphi}(0)=\IM\varphi(0)$, because $c\in (-1,1)$.
One can check that $\IM(i T_c(\tfrac{\bar a}{|a|}\lambda))\in\BLHL{}\,(0,\infty)$ for every $\lambda\in\TT$, and hence $\widetilde{\varphi}$ sends the arc $\lbrace\lambda\in\TT:\BLHL{}>0\rbrace$ to the line $1+i\,\RR$ and the arc $\lbrace\lambda\in\TT:\BLHL{}<0\rbrace$ to the line $i\,\RR$.
This gives $\RE\varphi^*=\RE\widetilde{\varphi}^*$ a.e. on $\TT$, so $\varphi=\widetilde{\varphi}$, as both of them are of class $H^1$ and $\IM\varphi(0)=\IM\widetilde{\varphi}(0)$.

In particular, we have the equality
\begin{equation}\label{eq_ex_s_compl_geod_fi_h}
\varphi_h(\lambda):=\frac{1}{2\pi}\int_{\lbrace\zeta\in\TT:\,\bar\zeta h(\zeta)>0\rbrace}\frac{\zeta+\lambda}{\zeta-\lambda}d\LEBT(\zeta) = \tau\left(i T_c(\tfrac{\bar a}{|a|}\lambda)\right),\;\lambda\in\DD.
\end{equation}
We shall use it in Example \ref{ex_klis_domain}.
Let us recall that the above equality holds for every $h:\CC\to\CC$ of the form $\bar a\lambda^2+b\lambda+a$, $a\in\CC$, $b\in\RR$, with $|b|<2|a|$, or equivalently: $\LEBT(\lbrace\lambda\in\TT:\BLHL{}>0\rbrace)\in(0,2\pi)$.
\end{EX}

\begin{LEM}\label{lemat_odcinki}
Let $D\subset\CC^n$ be a taut convex tube and let $V$ be a real $m\times n$ matrix with linearly independent rows $v_1,\ldots,v_m\in\RR^n$, $m\geq 1$, such that the domain $$\widetilde{D}:=\lbrace V\cdot z:z\in D\rbrace$$ is a taut convex tube in $\CC^m$.
\begin{enumerate}
\renewcommand{\theenumi}{(\roman{enumi})}
\renewcommand{\labelenumi}{\theenumi}
\item\label{lemat_odcinki_1} Let $\varphi:\DD\to D$ be a complex geodesic for $D$, and let $h$ be as in Theorem \ref{th_wkw_miara}.
If $\BLHL{}\in\SPAN_{\RR}\lbrace v_1,\ldots,v_m\rbrace$ for every $\lambda\in\TT$, then the mapping $\widetilde{\varphi}:\lambda\mapsto V\cdot\varphi(\lambda)$ is a complex geodesic for $\widetilde{D}$.
\item\label{lemat_odcinki_2} If for a holomorphic map $\varphi:\DD\to D$ the mapping $\widetilde{\varphi}:\lambda\mapsto V\cdot\varphi(\lambda)$ is a complex geodesic for $\widetilde{D}$, then $\varphi$ is a complex geodesic for $D$.
\end{enumerate}
\end{LEM}

The above lemma let us 'decrease' the dimension $n$, when we are trying to find a formula for $\varphi$, provided that the functions $h_1,\dots,h_n$ are linearly dependent.
In such situation, if we know formulas for geodesics in $\widetilde{D}$ (e.g. for $m=1$, because then $\widetilde{D}$ is a strip or a half-plane in $\CC$), then by \ref{lemat_odcinki_1} we obtain some information about $\varphi$, and by \ref{lemat_odcinki_2} we conclude that we cannot get anything more if we have no additional knowledge.
We use this lemma in Example \ref{ex_klis_domain} for $D$ with $\partial\RE D$ consisting of segments and half-lines.

The situation when $h_1,\dots,h_n$ are linearly dependent occurs e.g. when for some proper affine subspace $W$ of $\RR^n$ there is $\RE\varphi^*(\lambda)\in \INT_W(W\cap\partial\RE D)$ on the set of positive $\LEBT$ measure ($\INT_W$ denotes the interior with respect to $W$; if the set $\INT_W(W\cap\partial\RE D)$ is non-empty, then $W\cap\RE D=\varnothing$), because then the vectors $\BLHL{}$ are orthogonal to $W$ on the set of positive measure and hence on whole $\TT$ (by the identity principle).

Let us note that in the situation as in \ref{lemat_odcinki_2} the map $\varphi$ admits in fact a left inverse defined on the convex tube domain $\lbrace z\in\CC^n: V\cdot z\in\widetilde{D}\rbrace$, which may be larger than $D$ and not taut (its base may contain real lines).

\begin{proof}
We prove the first part.
The matrix $V^{T}$ may be viewed as a complex linear isomorphism from $\CC^m$ to $\SPAN_{\CC}\lbrace v_1,\ldots,v_m\rbrace$.
The mapping $\widetilde{h}:\CC\to\CC^m$ defined as $\widetilde{h}(\lambda)=\left(V^{T}\right)^{-1}\cdot h(\lambda)$ is of the form $\bar a\lambda^2+b\lambda+a$ (with some $a\in\CC^m$, $b\in\RR^m$) and it satisfies $\widetilde{h}(\lambda)^{T}\cdot V=h(\lambda)^{T}$ and $\widetilde{h}\not\equiv 0$.
We are going to apply Theorem \ref{th_wkw_miara} for $\widetilde{\varphi}$, $\widetilde{D}$, $\widetilde{h}$.
By weak-* limit argument, the boundary measure $\widetilde{\mu}$ of $\widetilde{\varphi}$ equals $V\cdot d\mu$.
For any $z\in D$ there is
\begin{eqnarray*}
\bar\lambda\widetilde{h}\left(\lambda)\bullet(V\cdot\RE z\,d\LEBT(\lambda)-d\widetilde{\mu}(\lambda)\right) &=& \bar\lambda\widetilde{h}(\lambda)^{T}\cdot V\cdot(\RE z\,d\LEBT(\lambda)-d\mu(\lambda))\\
&=& \BLHL{}\bullet(\RE z\,d\LEBT(\lambda)-d\mu(\lambda)),
\end{eqnarray*}
what is a negative measure.

To prove the second part it suffices to observe that if $f:\widetilde{D}\to\DD$ is a left inverse for $\widetilde{\varphi}$, then the map $z\mapsto f(V\cdot z)$, defined on the domain $\lbrace z\in\CC^n: V\cdot z\in\widetilde{D}\rbrace\supset D$, is a left inverse for $\varphi$.
\end{proof}

\begin{LEM}\label{lemat_kanty}
Let $D\subset\CC^n$ be a taut convex tube and let $p\in\partial\RE D$.
Define $$V:=\lbrace v\in\RR^n:\langle \RE z-p,v\rangle<0,\,z\in D\rbrace.$$
Let $\varphi:\DD\to D$ be a complex geodesic for $D$ with the boundary measure $\mu$ and let $h$ be as in Theorem \ref{th_wkw_miara}.
Put $$A:=\lbrace \lambda\in\TT:\BLHL{}\in\INT V\rbrace.$$
Then $$\chi_{A}\,d\mu = p\chi_{A}\,d\LEBT,$$
and $$\RE \varphi^*(\lambda)=p\text{ for every }\lambda\in A.$$
\end{LEM}

In the situation as in the above lemma, if $\INT V\neq\varnothing$, then we can say that $\RE D$ has a 'vertex' at the point $p$.
The aim of the lemma is to handle the situation when $\RE\varphi^*$ sends some $\lambda$'s to that vertex.
To detect some (not all) of those $\lambda$'s we analyse behaviour of the function $h$ instead of analysing behaviour of $\RE\varphi^*$; all $\lambda$'s detected in this way form the set $A$ (this is the reason for which in the definition of $A$ there is $\INT V$, not $V$ itself - for $\lambda$ such that $\BLHL{}\in V\setminus\INT V$ it is possible that $\RE\varphi^*(\lambda)\neq p$).
This approach let us state not only that $\RE\varphi^*(\lambda)=p$ for $\lambda\in A$, but much more: the boundary measure of $\varphi$ is equal to $p\,d\LEBT$ on the set $A$.
This lemma plays a key role in Example \ref{ex_klis_domain}.

If the set $\INT V$ is not empty, then it is an open, convex, infinite cone with the vertex at $0$.
In the case $n=2$ one can find two vectors $v_1,v_2\in\RR^n$, such that $\INT V$ consists of those $v\in\RR^n$, which lies 'between' $v_1$ and $v_2$, i.e. $\INT V=\lbrace v\in\RR^n:\det[v_1,v],\det[v,v_2]>0\rbrace$.

In the definition of $A$ the set $\INT V$ cannot be replaced by $V$, because then the equality $\chi_{A}\,d\mu = p\chi_{A}\,d\LEBT$ does not longer hold.
For example, take $D=\HHL^2$, $p=(0,0)$, and let $\varphi$ be given by the measure $-(\delta_1,\delta_1+\delta_{-1})$, that is $\varphi(\lambda)=\tfrac{1}{2\pi}\left(\tfrac{\lambda+1}{\lambda-1},\tfrac{\lambda+1}{\lambda-1}+\tfrac{\lambda-1}{\lambda+1}\right)$.
The map $\varphi$ is clearly a geodesic for $D$ and one can check that if $h=(h_1,h_2)$ is as in Theorem \ref{th_wkw_miara}, then $h_1(1)=0$ and $h_2\equiv 0$ (because $\BLHL{2}\geq 0$ on $\TT$ and $h$ has roots on $\TT$ at $1$ and $-1$).
As $V=[0,\infty)^2\setminus\lbrace (0,0)\rbrace$, we have $\lbrace\lambda\in\TT:\BLHL{}\in V\rbrace=\TT\setminus\lbrace 1\rbrace$, while the measure $\mu$ is clearly not equal to $(0,0)$ on $\TT\setminus\lbrace 1\rbrace$.

\begin{proof}
We may assume that $\INT V\neq\varnothing$.
For linearly independent vectors $v_1,\ldots,v_n\in\INT V$ set $$Q_{v_1,\ldots,v_n}:=\lbrace \alpha_1 v_1+\ldots+\alpha_n v_n:\alpha_1,\ldots,\alpha_n>0\rbrace.$$
One can check that the sets $Q_{v_1,\ldots,v_n}$ form an open covering of $\INT V$, and hence it suffices to show the conclusion with the set $\INT V$ replaced by $Q_{v_1,\ldots,v_n}$.

Fix $v_1,\ldots,v_n\in\INT V$ linearly independent, and let $W$ be a non-singular, real $n\times n$ matrix with rows $v_1,\ldots,v_n$.
Set $Q:=Q_{v_1,\ldots,v_n}$, $B:=\lbrace \lambda\in\TT:\BLHL{}\in Q\rbrace.$
We are going to show that $\chi_B\,d\mu=p\chi_B\,d\LEBT$ and $\RE\varphi^*=p$ on $B$.
Let $$\widetilde{\mu}:=(\widetilde{\mu}_1,\ldots,\widetilde{\mu}_n):=W\cdot (d\mu-p\,d\LEBT).$$
As $\widetilde{\mu}_j=\langle d\mu-p\,d\LEBT,v_j\rangle$, the measures $\widetilde{\mu}_j$ are negative.
The mapping $$\widetilde{h}(\lambda):=(\widetilde{h}_1(\lambda),\ldots,\widetilde{h}_n(\lambda)):=(W^{-1})^T \cdot h(\lambda),\,\lambda\in\CC,$$
satisfies $\bar\lambda \widetilde{h}(\lambda)\in (0,\infty)^n$ for $\lambda\in B$, because $(W^T)^{-1}\cdot(\alpha_1 v_1+\ldots+\alpha_n v_n)=(\alpha_1,\ldots,\alpha_n)$.
Thus
$$\chi_B(\lambda)\,\bar\lambda\widetilde{h}(\lambda)\bullet d\widetilde{\mu}(\lambda)\leq 0.$$
By the definition of $\widetilde{\mu}$ and $\widetilde{h}$ there is
$$\bar\lambda\widetilde{h}(\lambda)\bullet(W\cdot(\RE z-p)\,d\LEBT(\lambda)-d\widetilde{\mu}(\lambda)) = \BLHL{}\bullet(\RE z\,d\LEBT(\lambda)-d\mu(\lambda)),\;z\in D,$$
so the measure $\chi_B(\lambda)\,\bar\lambda\widetilde{h}(\lambda)\bullet(W\cdot(\RE z-p)\,\DLEBT(\lambda)-d\widetilde{\mu}(\lambda))$ is negative for every $z\in D$.
Tending with $z$ to $p$ we obtain
$$\chi_B(\lambda)\,\bar\lambda\widetilde{h}(\lambda)\bullet d\widetilde{\mu}(\lambda)\geq 0.$$
In summary, the measure $\chi_B(\lambda)\,\bar\lambda\widetilde{h}(\lambda)\bullet d\widetilde{\mu}(\lambda)$ is null.
As it is the sum of the negative measures $\chi_B(\lambda)\,\bar\lambda\widetilde{h}_j(\lambda)d\widetilde{\mu}_j(\lambda)$, all of them are null, and hence all $\chi_B d\widetilde{\mu}_j$ are also null.
Therefore $$\chi_B\,d\mu = W^{-1}\cdot \chi_B\,d\widetilde{\mu}+p\chi_B\,d\LEBT = p\chi_B\,d\LEBT,$$
so the first part is proved.

For the second, by the Poisson formula for $\varphi-p$ we have
$$\RE\varphi(\lambda)-p = \frac{1}{2\pi}\int_{\TT\setminus B}\frac{1-|\lambda|^2}{|\zeta-\lambda|^2}d(\mu-p\LEBT)(\zeta),\;\lambda\in\DD,$$
so $\RE\varphi(r\lambda)\to p$ as $r\to 1^-$, for any $\lambda\in B$.
\end{proof}

\begin{EX}\label{ex_klis_domain}
Consider a convex tube domain $D$ with $\RE D$ contained in $(-\infty,0)^2$ and $\partial\RE D$ being a sum of a horizontal half-line contained in $(-\infty,0]\times\lbrace 0\rbrace$, a vertical half-line contained in $\lbrace 0\rbrace\times[-\infty,0)$, and some finite number of segments.

More formally: let
$$D:=\lbrace z\in\CC^2:\langle \RE z-p_j, v_j\rangle<0\text{ for }j=1,\ldots,m\rbrace,$$ where $m\geq 2$, $v_1,\ldots,v_m\in[0,\infty)^2$, $p_0,\ldots,p_m\in(-\infty,0]^2$, $v_j=(v_{j,1},v_{j,2})$, $p_j=(p_{j,1},p_{j,2})$ are such that:
\begin{itemize}
\item $0=p_{0,1}=p_{1,1}>p_{2,1}>\ldots>p_{m-1,1}>p_{m,1}$,
\item $0=p_{m,2}=p_{m-1,2}>p_{m-2,2}>\ldots>p_{1,2}>p_{0,2}$,
\item $\det\left[v_j, v_{j+1}\right]>0$ for $j=1,\ldots,m-1$,
\item $\langle p_{j+1}-p_j,v_{j+1}\rangle = 0$ for $j=0,\ldots,m-1$
\end{itemize}
(the points $p_0$ and $p_m$ play only a supporting role).
The base of $D$ is shown on Figure \ref{fig_klis_domain}.
By the assumptions we have:
\begin{itemize}
\item $\langle \RE z-p_j, v_{j+1}\rangle<0$ for $z\in D$, $j=0,\ldots,m-1$,
\item $v_{1,1}>0$, $v_{1,2}=0$, $v_{m,1}=0$, $v_{m,2}>0$,
\item $\partial\RE D=\lbrace 0\rbrace\times(-\infty,p_{1,2}]\cup\bigcup_{j=1}^{m-2}[p_j,p_{j+1}]\cup(-\infty,p_{m-1,1}]\times\lbrace 0\rbrace$.
\end{itemize}

\begin{figure}[htb]
\centering
\includegraphics[scale=0.4]{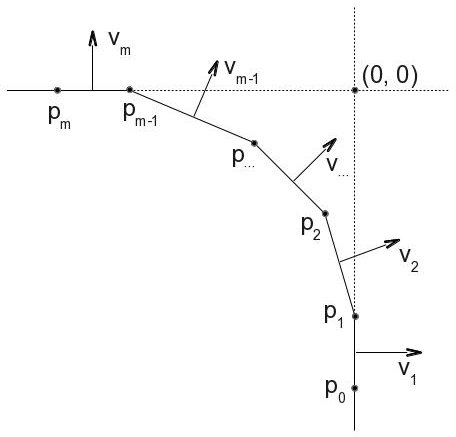}
\caption{\protect\label{fig_klis_domain}The base of $D$}
\end{figure}

Let $\varphi\in\OO(\DD,D)$ be a complex geodesic and let $\mu=(\mu_1,\mu_2)$ be its boundary measure.
Choose $h$ as in Theorem \ref{th_wkw_miara}, i.e. $h(\lambda)=\bar a\lambda^2+b\lambda+a$ with $a=(a_1,a_2)\in\CC^2$, $b=(b_1,b_2)\in\RR^2$, $h=(h_1,h_2)$, $h\not\equiv 0$, such that
\begin{equation}\label{eq_ex_klis_h_theorem_recall}
\BLHL{}\bullet(\RE z\, d\LEBT(\lambda)-d\mu(\lambda))\leq 0,\;z\in D.
\end{equation}
For a.e. $\lambda\in\TT$ the vector $\BLHL{}$ is outward from $\RE D$ at the boundary point $\RE\varphi^*(\lambda)$, so $$\BLHL{l}\geq 0,\;\lambda\in\TT,l=1,2.$$
Set
\begin{eqnarray*}
A_j &:=& \lbrace \lambda\in\TT:\det\left[\BLHL{},v_j\right]<0<\det\left[\BLHL{},v_{j+1}\right]\rbrace,\;j=1,\ldots,m-1,\\
B_j &:=& \lbrace \lambda\in\TT: \det\left[\BLHL{},v_j\right]=0\rbrace,\;j=1,\ldots,m,\\
B &:=& \bigcup_{j=1}^m B_j.
\end{eqnarray*}
The sets $A_1,\ldots,A_{m-1},B$ are pairwise disjoint and there is $$B\cup\bigcup_{j=1}^{m-1} A_j =\TT,$$ because every non-zero vector from $[0,\infty)^2$ lies 'between' some $v_j,v_{j+1}$, or is parallel to some $v_j$.

If $\LEBT(B)>0$, then for some $j_0\in\lbrace 1,\ldots,m\rbrace$ there is $\LEBT(B_{j_0})>0$ and the identity principle gives $B_{j_0}=\TT$.
Applying part \ref{lemat_odcinki_1} of Lemma \ref{lemat_odcinki} for the $1\times 2$ matrix with the row $v_{j_0}$ we get that $\langle \varphi(\cdot)-p_{j_0}, v_{j_0}\rangle$ is a geodesic for $\HHL$.
In view of part \ref{lemat_odcinki_2} of that lemma, the condition obtained is sufficient for $\varphi$ is a geodesic, so it is nothing more to do in this case.

Consider the situation when $\LEBT(B)=0$; the set $B$ is then finite and $v_{j,2}h_1-v_{j,1}h_2\not\equiv 0$, what in particular gives $h_1\not\equiv 0$, $h_2\not\equiv 0$.
By the equation \eqref{eq_ex_klis_h_theorem_recall} we get that the measure $\chi_B(\lambda)\,\BLHL{}\bullet d\mu(\lambda)$ is positive ($\chi_B\,d\LEBT$ is null).
Since $\BLHL{l}\geq 0$ on $\TT$ and $\mu_l\leq 0$, we have $\chi_B(\lambda)\BLHL{l}d\mu_l(\lambda)\leq 0$ ($l=1,2$).
Hence, the measure $\chi_B(\lambda)\,\BLHL{}\bullet d\mu(\lambda)$ is negative and in summary it is null.
As it is equal to sum of the negative measures $\chi_B(\lambda)\BLHL{1}d\mu_1(\lambda)$ and $\chi_B(\lambda)\BLHL{2}d\mu_2(\lambda)$, both of them are null.
Each $h_l$ has at most one root on $\TT$ (counting without multiplicities), so $$\chi_B\,d\mu_l=\alpha_l\delta_{\lambda_l}$$ for some $\lambda_l\in\TT$, $\alpha_l\leq 0$, with $\alpha_l h_l(\lambda_l)=0$.

Applying Lemma \ref{lemat_kanty} for $D$, $p_j$, $\varphi$ and $h$ (the set $A$ from the lemma is exactly the set $A_j$) we obtain $$\chi_{A_j}d\mu=p_j\chi_{A_j}\DLEBT,\,j=1,\ldots,m-1.$$
Therefore
\begin{equation}\label{eq_ex_klis_mu_Aj}
\mu_l = \sum_{j=1}^{m-1} p_{j,l}\chi_{A_j}d\LEBT+\alpha_l\delta_{\lambda_l},\;l=1,2,
\end{equation}
because $\mu_l = \sum_{j=1}^{m-1} \chi_{A_j}d\mu_l+\chi_Bd\mu_l$.

At this point, using \eqref{eq_ex_klis_mu_Aj} and the Poisson formula we can express the map $\varphi$ as an integral (with parameters $a$, $b$, $\alpha_1$, $\alpha_2$, and up to an imaginary constant), but in fact it is possible to derive a direct formula for it with usage of the mappings $\varphi_h$ defined in the equation \eqref{eq_ex_s_compl_geod_fi_h} in Example \ref{ex_pas_w_c}.
For this, let
\begin{equation}\label{eq_ex_klis_Cj}
C_j:=\lbrace \lambda\in\TT:\det\left[\BLHL{},v_j\right]<0\rbrace,\;j=1,\ldots,m.
\end{equation}
We have $C_1\supset C_2\supset\ldots\supset C_m$.
The set $(C_j\setminus C_{j+1})\setminus A_j\subset B$ is of zero Lebesgue measure and $A_j\subset C_j\setminus C_{j+1}$, so $\chi_{A_j}d\LEBT=\chi_{C_j}d\LEBT-\chi_{C_{j+1}}d\LEBT$.
Moreover, $\LEBT(C_1)=2\pi$ and $\LEBT(C_m)=0$, because $C_1=\lbrace\lambda\in\TT:\BLHL{2}>0\rbrace$ and $C_m=\lbrace\lambda\in\TT:\BLHL{1}<0\rbrace=\varnothing$.
Thus, the formula \eqref{eq_ex_klis_mu_Aj} may be written as
\begin{equation}\label{eq_ex_klis_mu_Cj_1_m}
\mu_l = p_{1,l}\DLEBT + \sum_{j=2}^{m-1} (p_{j,l}-p_{j-1,l}) \chi_{C_j} d\LEBT + \alpha_l\delta_{\lambda_l},\;l=1,2.
\end{equation}
The measures $\chi_{C_j}\,d\LEBT$ induces complex geodesics in $\SS$, provided that $\LEBT(C_j)\in(0,2\pi)$, because $$C_j=\lbrace\lambda\in\TT:\bar\lambda(v_{j,1}h_2(\lambda)-v_{j,2}h_1(\lambda))>0\rbrace$$ (see Example \ref{ex_pas_w_c} for details).
Therefore, it is good to remove from the sum \eqref{eq_ex_klis_mu_Cj_1_m} those $j$ which does not satisfy this condition.
For this, set
\begin{equation}\label{eq_ex_klis_k1_k2}
k_1 := \max\lbrace j\geq 1: \LEBT(C_j)=2\pi\rbrace,\; k_2 := \min\lbrace j\leq m:\LEBT(C_j)=0\rbrace.
\end{equation}
There is $1\leq k_1 < k_2 \leq m$.
By \eqref{eq_ex_klis_mu_Cj_1_m} we obtain
\begin{equation}\label{eq_ex_klis_mu_Cj_k1_k2}
\mu_l = p_{k_1,l}\DLEBT + \sum_{j=k_1+1}^{k_2-1} (p_{j,l}-p_{j-1,l}) \chi_{C_j} d\LEBT + \alpha_l\delta_{\lambda_l},\;l=1,2
\end{equation}
(note that it is possible that the above sum is empty, i.e. that $k_1+1>k_2-1$).
Now, for $j\in\lbrace k_1+1,\ldots,k_2-1\rbrace$ we have $\LEBT(C_j)\in(0,2\pi)$, and the Poisson formula let us derive the following formula for $\varphi$:
\begin{equation}\label{eq_ex_klis_summary}
\varphi_l(\lambda) = p_{k_1,l} + \sum_{j=k_1+1}^{k_2-1}(p_{j,l}-p_{j-1,l})\varphi_{v_{j,1}h_2-v_{j,2}h_1}(\lambda) + \frac{\alpha_l}{2\pi}\,\frac{\lambda_l+\lambda}{\lambda_l-\lambda}+i\beta_l,\;l=1,2,
\end{equation}
where $\beta_1$, $\beta_2$ are some real constants and $\varphi_{v_{j,1}h_2-v_{j,2}h_1}$ are as in \eqref{eq_ex_s_compl_geod_fi_h}, i.e.
$$\varphi_{v_{j,1}h_2-v_{j,2}h_1}(\lambda) = \tau\left(i T_{c_j}\left(\tfrac{\overline{v_{j,1}a_2-v_{j,2}a_1}}{|v_{j,1}a_2-v_{j,2}a_1|}\lambda\right)\right),\;\lambda\in\DD,$$
with $\tau(\lambda)=-\frac i\pi \log\left(i\,\frac{1+\lambda}{1-\lambda}\right)$ and $$c_j=\frac{-(v_{j,1}b_2-v_{j,2}b_1)}{2|v_{j,1}a_2-v_{j,2}a_1|+\sqrt{4|v_{j,1}a_2-v_{j,2}a_1|^2-(v_{j,1}b_2-v_{j,2}b_1)^2}}$$
(note that for $j=k_1+1,\ldots,k_2-1$ there is $|v_{j,1}b_2-v_{j,2}b_1|<2|v_{j,1}a_2-v_{j,2}a_1|$, because $\LEBT(C_j)\in (0,2\pi)$, and hence $c_j$ and $\varphi_{v_{j,1}h_2-v_{j,2}h_1}$ are well-defined).

In summary, a holomorphic map $\varphi:\DD\to\CC^2$ is a complex geodesic for the domain $D$ iff one of the following conditions holds:
\begin{enumerate}
\renewcommand{\theenumi}{(\roman{enumi})}
\renewcommand{\labelenumi}{\theenumi}
\item\label{ex_klis_domains_summary_1} $\varphi(\DD)\subset D$ and for some $j\in\lbrace 1\ldots m\rbrace$ the map $\lambda\mapsto\langle \varphi(\lambda)-p_j, v_j\rangle$ is a complex geodesic for $\HHL$, or
\item\label{ex_klis_domains_summary_2} $\varphi(\DD)\subset D$ and the map $\varphi$ is of the form \eqref{eq_ex_klis_summary} with some $\lambda_1,\lambda_2\in\TT$, $\alpha_1,\alpha_2\leq 0$, $\beta_1,\beta_2\in\RR$, and a map $h=(h_1,h_2)$ of the form $\bar a\lambda^2+b\lambda+a$ with $a=(a_1,a_2)\in\CC^2$, $b=(b_1,b_2)\in\RR^2$, such that $\BLHL{1}, \BLHL{2}\geq 0$ on $\TT$, $\alpha_1 h_1(\lambda_1)=\alpha_2 h_2(\lambda_2)=0$, $v_{j,1}h_2-v_{j,2}h_1\not\equiv 0$ for any $j=1\ldots,m$, where $k_1,k_2$ are given by \eqref{eq_ex_klis_k1_k2} with $C_j$ given by \eqref{eq_ex_klis_Cj}.
\end{enumerate}

So far, we have proved only that if $\varphi$ is a complex geodesic for $D$, then it satisfies one of the above conditions.
We are going to show the opposite implication now.
Take a holomorphic map $\varphi:\DD\to\CC^2$.
If $\varphi$ satisfies \ref{ex_klis_domains_summary_1}, then Lemma \ref{lemat_odcinki} does the job, so consider the situation as in \ref{ex_klis_domains_summary_2}.
As $\varphi(\DD)\subset D$, clearly $\varphi$ admits a boundary measure $\mu=(\mu_1,\mu_2)$.
There holds \eqref{eq_ex_klis_summary}, which gives \eqref{eq_ex_klis_mu_Cj_k1_k2} and hence \eqref{eq_ex_klis_mu_Cj_1_m}.
As $v_{j,1}h_2-v_{j,2}h_1\not\equiv 0$ for any $j$, the set $B$ is of $\LEBT$ measure $0$, so $\chi_{A_j}=\chi_{C_j}-\chi_{C_{j+1}}$ a.e. on $\TT$ (with respect to $\LEBT$).
Thus, \eqref{eq_ex_klis_mu_Cj_1_m} implies \eqref{eq_ex_klis_mu_Aj}.
From the equality \eqref{eq_ex_klis_mu_Aj} it follows that
\begin{equation}\label{eq_ex_klis_reverse_implication_Aj_B}
\chi_{A_j}d\mu_l = p_{j,l}\chi_{A_j}d\LEBT\text{ and }\chi_B d\mu_l = \alpha_l\delta_{\lambda_l}\text{ for } j=1,\ldots,m-1, l=1,2.
\end{equation}
Indeed, since $\TT$ is equal to sum of the pairwise disjoint sets $A_1,\ldots,A_{m-1},B$, the first statement is obvious, and for the second observe that if $\alpha_l=0$, then we are done, and if $\alpha_l<0$, then $h_l(\lambda_l)=0$, so $\lambda_l\not\in A_j$ for any $j$ and hence $\lambda_l\in B$.

If we show that for every set $E\in\lbrace A_1,\ldots,A_{m-1},B\rbrace$ and every point $z\in D$ the measure $$\BLHL{}\bullet(\RE z\,\chi_E(\lambda)\,\DLEBT(\lambda)-\chi_E(\lambda)\,d\mu(\lambda))$$ is negative, then we are done via Theorem \ref{th_wkw_miara}.

If $E=B$, then $\chi_E\,d\LEBT$ is a null measure and as $\BLHL{l}\alpha_l\,d\delta_{\lambda_l}(\lambda)=0$, $l=1,2$, by \eqref{eq_ex_klis_reverse_implication_Aj_B} the measure $\BLHL{}\bullet \chi_E(\lambda)\,d\mu(\lambda)$ is also null.

If $E=A_j$ for some $j=1,\ldots,m-1$, then by \eqref{eq_ex_klis_reverse_implication_Aj_B} we need to show that the measure $\BLHL{}\bullet(\RE z-p_j)\,\chi_{A_j}(\lambda)\,d\LEBT(\lambda)$ is negative for every $z\in D$.
But if $\lambda\in A_j$, then the vector $\BLHL{}$ lies 'between' $v_j$ and $v_{j+1}$, so $\BLHL{}=\gamma_1 v_j + \gamma_2 v_{j+1}$ for some $\gamma_1,\gamma_2\geq 0$ and hence $\BLHL{}\bullet(\RE z-p_j)\leq 0$.

Therefore, we proved that complex geodesics for $D$ are exactly the mappings of the form \ref{ex_klis_domains_summary_1} or \ref{ex_klis_domains_summary_2}.
\end{EX}

At the end, we present a simple example of convex tube domain with bounded base.
Here, the condition with radial limits (Theorem \ref{th_wkw_radial_limits}) suffices to obtain a direct formula for the real part of a geodesic $\varphi$, as its boundary measure is just $\RE\varphi^*d\LEBT$.

\begin{EX}\label{ex_kolo}
Let $$D=\lbrace (z_1,z_2)\in\RR^2: \left(\RE z_1\right)^2 + \left(\RE z_2\right)^2 < 1\rbrace.$$
Let $\varphi:\DD\to D$ be a complex geodesic and let $h$ be as in Theorem \ref{th_wkw_radial_limits}.
For a.e. $\lambda\in\TT$ the vector $\BLHL{}$ is a normal vector to $\partial\RE D$ at the point $\RE\varphi^*(\lambda)\in\partial\RE D$, so $\BLHL{}\in[0,\infty)\,\RE\varphi^*(\lambda)$.
As $\|\RE\varphi^*(\lambda)\|=1$ (we mean the euclidean norm), we get $$\RE\varphi^*(\lambda) = \frac{\BLHL{}}{\left\|\BLHL{}\right\|}\;\text{ for a.e. }\lambda\in\TT.$$
The map $h$ is of the form $\bar a\lambda^2+2b\lambda+a$ with $a\in\CC^n$, $b\in\RR^n$, $(a,b)\neq(0,0)$, so
\begin{equation}\label{eq_ex_kolo_re_fi}
\RE\varphi^*(\lambda) = \frac{\RE(\bar a\lambda)+b}{\left\|\RE(\bar a\lambda)+b\right\|},\;\text{ a.e. }\lambda\in\TT.
\end{equation}
As the boundary measure of $\varphi$ equals $\RE\varphi^*d\LEBT$, the Poisson formula let us derive an integral formula for $\varphi$.

On the other hand, by a similar reasoning one can show that any $\varphi\in\OO(\DD,D)$ satisfying \eqref{eq_ex_kolo_re_fi} with some $a\in\CC^n$, $b\in\RR^n$, $(a,b)\neq(0,0)$, is a complex geodesic for $D$.
\end{EX}

\section{Appendix}\label{sect_appendix}

In this section we are going to complete the proof of Proposition \ref{prop_warunki_wystarczajace}.
We start with a lemma, which gives Proposition \ref{prop_warunki_wystarczajace} as a corollary.
Its proof is strongly based on the proof of \cite[Lemma 8.2.2]{jarnickipflug}.
It is worth to point out that the lemma works for any domain $D$ in $\CC^n$, not necessarily tube.

\begin{LEM}\label{lem_ogolne_warunki_wystarczajace}
Let $D\subset\CC^n$ be a domain and let $\varphi:\DD\to D$ be a holomorphic map.
Suppose that there exists a map $h\in\OO(\DD,\CC^n)$ such that $\RE\left[h(0)\bullet\varphi'(0)\right]\neq 0$ and for every $z\in D$ the function $\psi_z\in\OO(\DD,\CC)$ defined as
\begin{equation*}
\psi_z(\lambda) := \frac{\varphi(0)-\varphi(\lambda)}{\lambda}\bullet h(\lambda)+\frac{h(\lambda)-h(0)}{\lambda}\bullet(z-\varphi(0))+\lambda\,\overline{h(0)\bullet(z-\varphi(0))},\;\lambda\in\DD_*
\end{equation*}
(and extended holomorphically through the origin) satisfies $$\RE\psi_z(\lambda)\leq 0,\;\lambda\in\DD.$$
Then the map $\varphi$ admits a left inverse on $D$.
\end{LEM}

The functions $\psi_z$ are defined same as in the proof of Lemma \ref{lem_wkw_h_fi_miara}.
The assumption that $\RE\psi_z(\lambda)\leq 0$ for all $z\in D$ and $\lambda\in\DD$ is clearly equivalent to the assumption that every $\RE\psi_z$ is bounded from above and $\RE\psi_z^*(\lambda)\leq 0$ for all $z\in D$ and a.e. $\lambda\in\TT$.
The last two conditions correspond to the conditions \ref{prop_warunki_wystarczajace_2} and \ref{prop_warunki_wystarczajace_1} of Proposition \ref{prop_warunki_wystarczajace}.
The reason for which Lemma \ref{lem_ogolne_warunki_wystarczajace} is not formulated in the same way as Proposition \ref{prop_warunki_wystarczajace} is to avoid using radial limits of $\varphi$, as they do not necessarily exist.

Proposition \ref{prop_warunki_wystarczajace} follows indeed from Lemma \ref{lem_ogolne_warunki_wystarczajace}, because if $D$, $\varphi$, $h$ are as in the proposition, then we have
$$\RE\psi_z^*(\lambda)=\RE\left[\BLHL{}\bullet(z-\varphi^*(\lambda))\right]\text{ for all }z\in D\text{ and a.e. }\lambda\in\TT$$
(the radial limits of $\varphi$ exist, as $D$ is a taut convex tube), so the assumptions of Lemma \ref{lem_ogolne_warunki_wystarczajace} are fulfilled.

\begin{proof}
For $\epsilon\geq 0$ define
\begin{eqnarray*}
\Phi_\epsilon(z,\lambda) &=&(z-\varphi(\lambda))\bullet h(\lambda)-\epsilon\lambda,\;z\in\CC^n,\lambda\in\DD,\\
\Psi_\epsilon(z,\lambda) &=&\frac1\lambda \Phi_\epsilon(z,\lambda),\;z\in\CC^n,\lambda\in\DD_*.
\end{eqnarray*}
We have $$\Psi_\epsilon(\varphi(0),\lambda) = \psi_{\varphi(0)}(\lambda)-\epsilon,\;\lambda\in\DD_*,\epsilon\geq 0.$$
The function $\Psi_\epsilon(\varphi(0),\cdot)$ extends holomorphically through $0$ and
$\RE\Psi_\epsilon(\varphi(0),\cdot)\leq-\epsilon$ on $\DD$.
Moreover, by the assumption $\RE\left[h(0)\bullet\varphi'(0)\right]\neq 0$ we get $\RE\Psi_0(\varphi(0),\cdot)<0$ on $\DD$.
In summary, for any $\epsilon\geq 0$ there is $\RE\Psi_\epsilon(\varphi(0),\cdot)<0$ on $\DD$, so the point $0$ is the only root of $\Phi_\epsilon(\varphi(0),\cdot)$ on $\DD$ and it is a simple root.

Assume for a moment that
\begin{equation}\label{eq_loww_istnieje_rozwiazanie_1}
\text{there exists }f\in\OO(D,\DD)\text{ such that }\Phi_0(z,f(z))=0,\;z\in D.
\end{equation}
We claim that $f$ is a left inverse for $\varphi$.
Since $\Phi_0(\varphi(0),0)=0$ and $\PDER{\Phi_0}{\lambda}(\varphi(0),0)=\Psi_0(\varphi(0),0)\neq 0$,
by the implicit mapping theorem there is an open neighbourhood $U\subset D\times\DD$ of $(\varphi(0),0)$ such that $U\cap\Phi_0^{-1}(0)$ is equal to the graph of some holomorphic function of the variable $z$ defined near $\varphi(0)$ and mapping $\varphi(0)$ to $0$.
We have $f(\varphi(0))=0$, because $0$ is the only root of $\Phi_0(\varphi(0),\cdot)$, and hence $(z,f(z))\in U$ for $z$ near $\varphi(0)$.
Let $\Gamma$ be the graph of $f$.
Clearly $U\cap\Gamma\subset U\cap\Phi_0^{-1}(0)$, so shrinking $U$ if necessary we obtain $U\cap\Gamma=U\cap\Phi_0^{-1}(0)$.
As $\Phi_0(\varphi(\lambda),\lambda)=0$ on $\DD$, there is $(\varphi(\lambda),\lambda)\in\Gamma$ for $\lambda$ near $0$.
This implies $f(\varphi(\lambda))=\lambda$ near $0$ and hence on the whole $\DD$.

It remains to prove \eqref{eq_loww_istnieje_rozwiazanie_1}, and for this it suffices to show that
\begin{equation}\label{eq_loww_istnieje_rozwiazanie_2}
\text{for any }\epsilon>0\text{ there exists }f_\epsilon\in\OO(D,\DD)\text{ such that }\Phi_\epsilon(z,f_\epsilon(z))=0,\;z\in D.
\end{equation}
Indeed, using the Montel theorem choose a sequence $(f_{\epsilon_k})_k$ ($\epsilon_k\to 0$ as $k\to\infty$) convergent to a holomorphic function $f:D\to\CC$.
As $0$ is the only root of $\Phi_\epsilon(\varphi(0),\cdot)$, we get $f_\epsilon(\varphi(0))=0$ and hence $f(D)\subset\DD$, what let us easily derive \eqref{eq_loww_istnieje_rozwiazanie_1}.

The statement \eqref{eq_loww_istnieje_rozwiazanie_2} follows from the following claim:
\begin{equation}\label{eq_loww_szacowanie_re_psi_eps}
\begin{array}{l}
\text{for every }\epsilon>0\text{ and }K\subset\subset D\text{ there exists }r\in(0,1)\text{ such that }\\
\RE\Psi_\epsilon(z,\lambda)<0\text{ for }z\in K, |\lambda|\in[r,1).
\end{array}
\end{equation}
Indeed, assume \eqref{eq_loww_szacowanie_re_psi_eps} and fix $\epsilon>0$.
Let $z\in D$ and let $r=r(\epsilon,z)$ be taken as above for $K=\lbrace z\rbrace$.
The function $\Phi_\epsilon(z,\cdot)$ has no roots in $\DD\setminus r\DD$, because $\RE\Psi_\epsilon(z,\lambda)<0$ for $|\lambda|\in[r,1)$.
Moreover,
\begin{equation}\label{eq_loww_jedyny_pierwiastek_istnienie}
\frac{1}{2\pi i}\int_{r\TT}\frac{\PDER{\Phi_\epsilon}{\lambda}(z,\lambda)}{\Phi_\epsilon(z,\lambda)}d\lambda = 1 + \frac{1}{2\pi i}\int_{r\TT}\frac{\PDER{\Psi_\epsilon}{\lambda}(z,\lambda)}{\Psi_\epsilon(z,\lambda)}d\lambda = 1
\end{equation}
(the last integral is just the index at $0$ of the curve $s\mapsto\Psi_\epsilon(z,re^{is})$, equal to $0$ by \eqref{eq_loww_szacowanie_re_psi_eps}), so $\Phi_\epsilon(z,\cdot)$ has only one root in $\DD$ (counting with multiplicities).
Denote this root by $f_\epsilon(z)$.
We have the function $f_\epsilon:D\to\DD$ such that $\Phi_\epsilon(z,f_\epsilon(z))=0$, so we only need to show that it is holomorphic.

Fix $K\subset\subset D$ and let $r=r(\epsilon,K)$ be as in \eqref{eq_loww_szacowanie_re_psi_eps}.
Again, $\Phi_\epsilon(z,\cdot)$ has no roots in $\DD\setminus r\DD$ for $z\in K$, so $f_\epsilon(K)\subset r\DD$.
As $f_\epsilon(z)$ is the only root of $\Phi_\epsilon(z,\cdot)$ and it belongs to $r\DD$, we have the formula
\begin{equation}\label{eq_loww_jedyny_pierwiastek_wzor}
f_\epsilon(z) = \frac{1}{2\pi i}\int_{r\TT}\lambda\frac{\PDER{\Phi_\epsilon}{\lambda}(z,\lambda)}{\Phi_\epsilon(z,\lambda)}d\lambda,\;z\in K,
\end{equation}
what implies that $f_\epsilon$ is holomorphic in $\INT K$.
As $K$ is arbitrary, we obtain $f_\epsilon\in\OO(D,\DD)$.

It remains to show \eqref{eq_loww_szacowanie_re_psi_eps}.
Fix $\epsilon>0$ and $K\subset\subset D$.
For $z\in K$ and $\lambda\in\DD_*$ we have $$\RE\Psi_\epsilon(z,\lambda) = \RE\psi_z(\lambda)+ \RE\left[\frac1\lambda\,h(0)\bullet(z-\varphi(0)) -\lambda\,\overline{h(0)\bullet(z-\varphi(0))}\right]-\epsilon.$$
The second term of the right hand side tends uniformly (w.r.t $z\in K$) to $0$ as $|\lambda|\to 1$, and the first term is non-positive.
This gives \eqref{eq_loww_szacowanie_re_psi_eps} and finishes the proof.
\end{proof}

%
%



%
%
%

\end{document}